\documentclass[12pt,reqno]{article}

\usepackage[usenames]{color}
\usepackage{amssymb}
\usepackage{graphicx}
\usepackage{amscd}

\usepackage[colorlinks=true,
linkcolor=webgreen,
citecolor=webgreen]{hyperref}

\definecolor{webgreen}{rgb}{0,.5,0}

\usepackage{color}
\usepackage{fullpage}
\usepackage{float}

\usepackage{amsmath,amssymb}
\usepackage{amsthm}
\usepackage{amsfonts}
\usepackage{latexsym}

\newcommand{\seqnum}[1]{\href{http://www.research.att.com/cgi-bin/access.cgi/as/~njas/sequences/eisA.cgi?Anum=#1}{\underline{#1}}}

\begin{document}

\theoremstyle{plain}
\newtheorem{theorem}{Theorem}
\newtheorem{corollary}[theorem]{Corollary}
\newtheorem{lemma}[theorem]{Lemma}
\newtheorem{proposition}[theorem]{Proposition}

\theoremstyle{definition}
\newtheorem{definition}[theorem]{Definition}
\newtheorem{example}[theorem]{Example}
\newtheorem{conjecture}[theorem]{Conjecture}

\theoremstyle{remark}
\newtheorem{remark}[theorem]{Remark}

\begin{center}
\vskip 1cm{\LARGE\bf An Analytical Approach to \\
\vskip .1in
Exponent-Restricted Multiple \\
\vskip .1in
Counting Sequences}
\vskip 1cm
\large
M. H\"{u}srev Cilasun\\
School of Electrical and Electronics\\
Istanbul Technical University\\
Maslak, Istanbul\\
Turkey\\
\href{mailto:cilasun@itu.edu.tr}{\tt cilasun@itu.edu.tr}
\end{center}

\vskip 0.2 in

\begin{abstract}
This study involves definitions for multiple-counting regular and summation sequences of $\rho$. My paper introduces and proves recurrent relationships for multiple-counting sequences and shows their association with Fermat's little theorem. I also studied matrix representations and obtained generalized Binet formulas for defined sequences. As a result of multiple-counting sequence explanation, this study leads to a better understanding for distribution of composite numbers between consecutive bases.
\end{abstract}

\section{Introduction}
Jacobsthal Sequence is defined as
\begin{equation}
j_{n+2}=j_{n+1}+2j_{n},j_{1}=1,j_{0}=0,n\geq 0  \label{eq:1}
\end{equation} \cite{ref1}
Horadam and Hoggatt has conducted many research on Jacobsthal Sequence. Jacobsthal Sequence has a lot of remarkable properties as counting microcontroller skip instructions \cite{ref6} and counting number of ways to tile a $3\times(n-1)$ rectangle with $2\times 2$ and $1\times 1$ tiles \cite{ref2} \emph{et cetera}. Barry \cite{ref2} also shows that Jacobsthal Sequence can be expressed in floor function notation.  
Fossaceca \cite{ref2} mentioned at sequence A001045 that each element of Jacobsthal Sequence counts integer multiples of $3$ in between corresponding exponents of $2$.

\begin{table}[h]
\label{tablo1} \centering \footnotesize {
\begin{tabular}{|c|c|c|c|c|c|c|c|c|c|c|c|}
\hline
Interval boundaries &  &  &  &  &  &  &  &  &  &  &\\
of exponents of 2 & $2^0$ &  & $2^1$ &  & $2^2$ &  & $2^3$ &  & $2^4$ &  & $%
2^5$ \\ \hline
Numbers of integer &  & 0 &  & 1 &  & 1 &  & 3 &  & 5& \\
multiples of 3 in each interval. &  &  &  &  &  &  &  &  &  &  &\\
\hline
\end{tabular}
} \caption{Number of integers which divides 3 with no remainder
between exponents of 2. The sequence is identical to Jacobsthal
Sequence.}
\end{table}

It seems that there are other Jacobsthal-like recurrence relationships held in between some other exponents and bases when multiplicative numbers are counted and this study investigates various aspects of those recurrent sequences. 

Another concept will be discussed in this paper is Fermat's little theorem's condition which is given in \cite{ref3} as $x^{\rho -1}\equiv$ 1(mod $\rho$). The theorem states that the given condition is satisfied when $\rho$ and $x$ are coprime and $\rho$ is prime. However, vice versa is not always true. The condition is still valid for some composite $\rho$s called \emph{Fermat pseudoprime}s, \emph{Carmichael number}s or \emph{$K_1$ Kn\"{o}del number}s. Particularly for $x=2$ case, $\rho$s are called \emph{Sarrus number}s, \emph{Poulet number}s or \emph{Fermatian}s if they satisfy Fermat's given condition \cite{ref7}.

\section{Generalized Multiple-Counting Sequences between
Exponents}

\begin{definition}
	Let $x$ an $\rho$ be elements of integer set $\{2,3,4,5\dots \}$ and $n$ be an element of $\{0,1,2,3\dots \}$. If an integer sequence $J$ is defined in which $J_{n}$ is equal to the number of multiples of $\rho $ which are greater than $x^{n}$ and less than $x^{n+1}$; then $J$ is a \textit{multiple-counting sequence} of $\rho$ between exponents of $x$.
	Using floor functions, we can derive number of multiples of $\rho$ less than $x^{n+1}$ as $\left \lfloor \frac{x^{n+1}}{\rho} \right \rfloor$. As a natural consequence, multiples of $\rho$ greater than $x^{n}$ and less than $x^{n+1}$ is found as follows:
\begin{equation} \label{floornot}
J_n = \left \lfloor \frac{x^{n+1}}{\rho} \right \rfloor - \left \lfloor \frac{x^{n}}{\rho} \right \rfloor
\end{equation}
\end{definition}

\begin{theorem} Let $x$ and $\rho $ be integers from set $\{1,2,3\dots
\}$. Let numbers $x $ and $\rho $ be relatively prime and satisfy
Fermat's $x^{\rho -1}\equiv $ 1(mod $\rho$ ) condition. For any number $n$ from set $\{0,1,2\dots \}$ with already provided starting terms
$J_{0},J_{1}\dots J_{\rho -1}$, the sequence $J_{n}$ which counts
multiples of $\rho $ between $x^{n}$ and $x^{n+1}$ , also satisfies
the following recurrence relationship:
\begin{equation}
J_{n+\rho -1}=(x-1)\sum_{i=1}^{\rho -2}{J_{n-i+\rho -1}}+xJ_{n}
\label{eq:2}
\end{equation}
\end{theorem}%
\begin{proof}
Let the expression is transformed arithmetically to the following:
\begin{equation}\label{eq:3}
J_{n+\rho-1}=(x-1)\sum^{\rho-1}_{i=1}{J_{n-i+\rho-1}}+J_n
\end{equation}
Aassume that $x^n$ is formulated parametrically by
$\lambda\rho+\varphi$, then $x^{n+w}$ is also formulated as
$x^w(\lambda\rho+\varphi)$. It is easily found that there are
$(x-1)(\lambda\rho+\varphi)-1$ integers between $x^n$ and $x^{n+1}$.
{\scriptsize
\begin{table}[h]\label{tablo2}
\centering \scriptsize
\begin{tabular}{|c|c|c|c|c|c|c|}
\hline
$x^n$&&$x^{n+1}$&\dots&$x^{n+\rho-1}$&&$x^{n+\rho}$\\
\hline
&$(x-1)(\lambda\rho+\varphi)-1$&&\dots
&&$x^{\rho-1}(x-1)(\lambda\rho+\varphi)-1$&\\
\hline
\end{tabular}
\caption{Numbers of integers between exponents of x.}
\end{table}}

Table \ref{tablo2} shows that there are $x^{\rho-1}(x-1)(\lambda\rho+\varphi)-1$ numbers in between $x^{n+\rho-1}$ and $x^{n+\rho}$. Let $\delta$ be a
unique arithmetical minimum plus or maximum minus process which
makes $(x-1)(\lambda\rho+\varphi)-1$ a certain multiple of $\rho$.
After $\delta$ process is applied, number of integers which divide $\rho$ in chosen interval is easily found by dividing the number of elements in chosen interval by $\rho$.
\begin{equation}\label{eq:4}
    J_n = \frac{(x-1)(\lambda\rho+\varphi)-1+\delta}{\rho}
\end{equation}
Because of $x^{\rho-1}\equiv 1(mod \rho)$  initial condition, the
expression $x^{\rho-1}(x-1)(\lambda\rho+\varphi)-1$ also needs the
identical $\delta$ process to be made a multiple of $\rho$.
\begin{equation}\label{eq:5}
    J_{n+\rho-1} = \frac{x^{\rho-1}(x-1)(\lambda\rho+\varphi)-1+\delta}{\rho}
\end{equation}
\begin{table}[h]\label{tablo3}
\centering \footnotesize
\begin{tabular}{|c|c|c|c|c|c|c|c|c|c|}
\hline
$(Mod \rho) \equiv$&$1$&$2$&$3$&$\dots$&$\varphi$&$\varphi +1$&$\dots$&$\rho-1$&$\rho$\\
\hline
&&&&&$\lambda\rho+\varphi$&*&*&*&*\\
&*&*&*&*&*&*&*&*&*\\
&\dots&\dots&\dots&\dots&\dots&\dots&\dots&\dots&\dots \\
&*&*&*&*&*&*&*&*&*\\
&*&*&*&*&$x^{\rho-1}(\lambda\rho+\varphi)$&&&&\\
\hline
\end{tabular}
\caption{\footnotesize Distribution for $\varepsilon$. Asterisk signs
refer to the consecutive integers between boundaries $
\lambda\rho+\varphi $ and $ x^{\rho-1}(\lambda\rho+\varphi)$. $\rho
$ rows include terms which are multiple of $\rho$.}
\end{table}

On the other hand, let $\varepsilon$ be another unique process which
makes the number of integers at given interval divides $\rho$. We see at Table 3 that $\varepsilon$ process must be equal to +1 if we want to find the number of integers which divides $\rho$. Then the sigma notation part of equation is found as follows:
\begin{equation}\label{eq:6}
(x-1)\sum^{\rho-2}_{i=1}{J_{n-i+\rho-1}}
=\frac{(x-1)(x^{\rho-1}(\lambda\rho+\varphi)-\lambda\rho-\varphi-1+\varepsilon)}{\rho}
\end{equation}
Rewrite expression (\ref{eq:3}) by combining (\ref{eq:4}), (\ref{eq:5}) and (\ref{eq:6}) and consider that $\varepsilon=+1$, then proof is provided after simplification. 
\end{proof}
 
\begin{lemma}
An alternative simpler proof for the previous theorem can be derived from the floor function notation of $J_n$ in (\ref{floornot}).
\end{lemma}

\begin{proof}
For a multiple-counting sequence which satisfies the recurrence relationship given in (\ref{eq:3}), it is also known that $x^{\rho - 1} \equiv 1 (mod \rho)$. Replace $J_n$ with its floor function notation (\ref{floornot}). The following expression is obtained when intermediate elements of the sigma notation are eliminated by themselves.
\begin{equation}\label{eq:e}
 \left \lfloor \frac{x^{n+\rho}}{\rho}\right\rfloor-\left\lfloor\frac{x^{n+\rho-1}}{\rho}\right\rfloor= (x-1)\left ( \left\lfloor\frac{x^{n+\rho-1}}{\rho}\right\rfloor -\left\lfloor\frac{x^n}{\rho}\right\rfloor\right ) +\left\lfloor \frac{x^n+1}{\rho}\right\rfloor-\left\lfloor\frac{x^n}{\rho}\right\rfloor
\end{equation}
Let $a$ be an arbitrary constant. Assume $x^{\rho-1}$ is equal to $a\rho+1$ as a result of the $x^{\rho - 1} \equiv 1 (mod \rho)$ condition. When $a\rho+1$ is substituted to the (\ref{eq:e}), floor function takes $a$-terms out because of its integer exclusion property. Right and left hand sides are simplified and proof is completed. 
\end{proof}

\begin{example}
Table \ref{tablo4} contains several examples of multiple-counting sequences for different $\rho$s and $x$s

\begin{table}[h]
\centering
\begin{tabular}{|c|c|c|}
\hline 
$\rho$ & $x$ &\textbf{Summary of Sequence} \\ 
\hline 
3 & 2 & {\footnotesize Identical to Jacobsthal Sequence: $J_n=J_{n-1}+2J_{n-2},n\geq 2, J_0=0, J_1=1$} \\
\hline 
5 & 3 & {\footnotesize $J_n=2J_{n-1}+2J_{n-2}+2J_{n-3}+3J_{n-4},n\geq 4, J_0=0, J_1=1, J_2=4, J_3=11$}\\
\hline
3 & 10 & {\footnotesize $J_n=9J_{n-1}+10J_{n-2},n\geq 2, J_0=3, J_1=30$}\\
&&  \footnotesize $J_n$ counts (n+1)-digit numbers which divides 3 \\
\hline
\end{tabular}
\caption{Some multiple-counting sequences} \label{tablo4}
\end{table}

Additionally, Cooka and Bacon \cite{ref8} defines sequences $J_n$s for various $\rho$s called \textquotedblleft Higher Order Jacobsthal Sequences\textquotedblright when $x=2$.

\end{example}

\begin{remark}
For any element of sequence $J_n$, Binet expression
\begin{equation}  \label{eq:7}
J_n=\sum^n_{i=1}{\lambda_i^{n-i+2}v_{1,i}(\sum_{j=1}^n{v_{i,j}^{-1}u_j^{%
\rho-2}})}
\end{equation}
is always satisfied when key matrix

\[
K_{(\rho-1) \times (\rho-1)}= \left(
\begin{array}{cccccc}
x-1 & x-1 & \dots & x-1 & x-1 & x \\
1 & 0 & \dots & 0 & 0 & 0 \\
0 & 1 & \ddots & 0 & 0 & 0 \\
\vdots & \ddots & \ddots & \ddots & \vdots & \vdots \\
0 & 0 & \ddots & 1 & 0 & 0 \\
0 & 0 & \dots & 0 & 1 & 0 \\
&  &  &  &  &
\end{array}
\right)
\]

satisfies equation $Ku^n=u^{n+1}$ for vector

\[
u^n= \left(
\begin{array}{c}
u_1 \\
u_2 \\
\vdots \\
u_{p-1} \\
\end{array}
\right) = \left(
\begin{array}{c}
J_n \\
J_{n-1} \\
\vdots \\
J_{n-p+2} \\
\end{array}
\right)
\]
if K is diagonalizable and K's eigenvector matrix V is non-singular.
Note that elements $u_{i}$,$v_{ij}$ belong to u,V.
\end{remark}

\begin{proof}
 Using Kalman's method \cite{ref5}, explicit Binet form of our sequence is easily obtained. After $K=V\Lambda V^{-1}$ decomposition;  equation $Ku^n=u^{n+1}$  is transformed to $V\Lambda V^{-1}u^n=u^{n+1}$ where $\Lambda$ is diagonal eigenvalue matrix of K. We only need top elements of vectors, thus s muiltiply the expression by
$$
\left(
\begin{array}{cccc}
1&0&\dots&0\\
\end{array}
\right)
$$
$(\rho-1)\times 1$ matrix from the left. Any particular term of sequence
$J_n$ satisfies equation:
\begin{equation} \label{regmat}
J_n= \left(
\begin{array}{cccc}
1&0&\dots&0\\
\end{array}
\right) (V\Lambda^{n-\rho+2}(V^{-1}u^{\rho-2}))
\end{equation}
Since
$$
(V^{-1}u^{\rho-2}) = \left(
\begin{array}{cccc}
v^{-1}_{1,1}&v^{-1}_{1,2}&\dots&v^{-1}_{1,\rho-1}\\
v^{-1}_{2,1}&v^{-1}_{2,2}&\dots&v^{-1}_{2,\rho-1}\\
\dots&\dots&\ddots&\dots\\
v^{-1}_{{\rho-1 },1}&v^{-1}_{{\rho-1},2}&\dots&v^{-1}_{\rho-1,\rho-1}\\
\end{array}
\right) \left(
\begin{array}{c}
J_{\rho-2}\\
J_{\rho-3}\\
\vdots\\
J_{0}\\
\end{array}
\right)
$$
 and
$$
 V =
\left(
\begin{array}{cccc}
v_{1,1}&v_{1,2}&\dots&v_{1,\rho-1}\\
v_{2,1}&v_{2,2}&\dots&v_{2,\rho-1}\\
\dots&\dots&\ddots&\dots\\
v_{{\rho-1 },1}&v_{{\rho-1},2}&\dots&v_{\rho-1,\rho-1}\\
\end{array}\\
\right)
$$
and
$$
\Lambda^{n-\rho+2} =
\left(
\begin{array}{cccc}
\lambda_{1}^{n-\rho+2}&0&\dots&0\\
0&\lambda_{2}^{n-\rho+2}&\ddots&\vdots\\
\vdots&\ddots&\ddots&0\\
0&\dots&0&\lambda_{\rho-1}^{n-\rho+2}\\
\end{array}
\right)\\
$$
are known, we can compute (\ref{regmat}) after required substitution.
Remember that $\lambda_1,\lambda_2, \dots \lambda_n$ denote
eigenvalues of K. After the outcome of the computation is simplified, we obtain the initial expression (\ref{eq:7}). 
\end{proof}

\vskip 0.7cm
\section{Generalized Multiple-Counting Summation Sequences
until Exponents}

\begin{definition}
Let an$x$ and $\rho$ be elements of integer set $A=\{2,3,4,5 \dots\}$ and n
be  element of $\{0, 1, 2, 3, \dots\}$. If an integer sequence $S$ is
defined when $S_n$ is equal to the number of integer multiples of
$\rho$, which are greater than $0$ and less than $x^{n+1}$; then $J$
is a \it{multiple-counting summation sequence} of $\rho$. Summation sequence has a more compact form as follows:
\begin{equation}
S_n =\left \lfloor \frac{x^{n+1}}{\rho} \right \rfloor
\end{equation}
The strong interrelation between regular multiple-counting sequences and \textit{multiple-counting summation sequences} is expressed as

\begin{equation}\label{eq:11}
S_n=\sum^{n}_{i=0}{J_i}
\end{equation}

\end{definition}

The summation idea of a multiple-counting sequence is very similar to the \textit{Jacobsthal Representation Sequence} which Horadam \cite{ref4} mentions. The summation sequence satisfies the same recurrence relationship which regular sequence also does, except for a particularly determined constant.

\begin{theorem}
Let x and $\rho$ be integers from set $\{1,2,3\dots\}$ and let
numbers x and $\rho$ satisfy Fermat's $x^{\rho-1}\equiv$ 1(mod $\rho$)
condition. For any number $n$ from set $R=\{0,1,2\dots\}$, with
already provided starting terms $S_0,S_1 \dots S_{p-1}$ , the
sequence $S_n$ which counts multiples of $\rho$ between $0$ and
$x^{n+1}$ , satisfies the following recurrence relationship:
\begin{equation}  \label{eq:8}
S_{n+\rho-1}=(x-1)\sum^{\rho-2}_{i=1}{S_{n-i+\rho-1}}+xS_n+\pi
\end{equation}
or
\begin{equation}  \label{eq:9}
S_{n+\rho-1}=(x-1)\sum^{\rho-1}_{i=1}{S_{n-i+\rho-1}}+S_n+\pi
\end{equation}
,where $\pi$ is a unique constant for a particular combination of
$\rho$,x and $J_0,J_1\dots J_{p-2}$ initial conditions. Formula for
$\pi$ is:
\begin{equation}  \label{eq:10}
\pi=\sum^{\rho-1}_{i=0}{J_i}-(x-1)\sum^{\rho-1}_{i=1}{\sum^{i-2}_{j=0}{J_j}
}
\end{equation}
\end{theorem}

\begin{proof}
Rewrite (\ref{eq:9}) using property (\ref{eq:11}):
\begin{equation}\label{eq:12}
\sum^{n+\rho-1}_{i=0}{J_i}=(x-1)\sum^{\rho-1}_{i=1}{\sum^{n-i+\rho-1}_{j=0}{J_j}}+\sum^{n}_{i=0}{J_i}+\pi
\end{equation}
Separate it into n-dependent and n-independent parts:

\begin{equation}\label{eq:13}
\sum^{\rho-2}_{i=0}{J_i}+\sum^{n+\rho-1}_{i=\rho-1}{J_i}=(x-1)\sum^{\rho-1}_{i=1}{(\sum^{i-2}_{j=0}{J_j}+\sum^{n-i+\rho-1}_{j=i-1}{J_j})}+\sum^{n}_{i=0}{J_i}+\pi
\end{equation}

\begin{equation}\label{eq:14}
\sum^{\rho-2}_{i=0}{J_i}+\underbrace{\sum^{n+\rho-1}_{i=\rho-1}{J_i}}=(x-1)\sum^{\rho-1}_{i=1}{\sum^{i-2}_{j=0}{J_j}}+\underbrace{(x-1)\sum^{\rho-1}_{i=1}{\sum^{n-i+\rho-1}_{j=i-1}{J_j}}}+\underbrace{\sum^{n}_{i=0}{J_i}}+\pi
\end{equation}

n-dependent underbraced parts of equation (\ref{eq:14}) are
eliminated by induction on (\ref{eq:3}). Since n-independent parts
of expression above are constant for whole sequence, $\pi$ also
depends on a constant value and required proof is provided.
\end{proof}

\begin{remark}
To obtain a matrix-based Binet representation for sequence $S$, we
use a similar approach with the sequence $J$. Since the only
difference between definitions of two sequences $J$ and $S$ is a
constant $\pi$, we can easily modify the key matrix to insert $\pi$
dependency with notational changes to get a simple formula.
\begin{equation}  \label{eq:15}
S_n=\sum^{n}_{i=1}{\mu^{n-i+2}_{i}w_{1,i}(\sum^{n}_{j=1}{w^{-1}_{i,j}%
\tau_{j}^{\rho-1}})}
\end{equation}
Assuming L is the such a key matrix
\[
L_{\rho \times \rho}= \left(
\begin{array}{ccccccc}
x-1 & x-1 & \dots & x-1 & x-1 & x & 1 \\
1 & 0 & \dots & 0 & 0 & 0 & 0 \\
0 & 1 & \ddots & 0 & 0 & 0 & 0 \\
\vdots & \ddots & \ddots & \ddots & \vdots & \vdots &  \\
0 & 0 & \ddots & 1 & 0 & 0 & 0 \\
0 & 0 & \dots & 0 & 1 & 0 & 0 \\
0 & 0 & 0 & 0 & 0 & 0 & 1 \\
&  &  &  &  &  &
\end{array}
\right)
\]
of $L\tau^n=\tau^{n+1}$ where $\tau$ holds equation:
\[
\tau^n= \left(
\begin{array}{c}
\tau_1 \\
\tau_2 \\
\vdots \\
\tau_{p-1} \\
\pi \\
\end{array}
\right) = \left(
\begin{array}{c}
S_n \\
S_{n-1} \\
\vdots \\
S_{n-p+2} \\
\pi \\
\end{array}
\right)
\]
if L is diagonalizable and L's eigenvector matrix W is non-singular.
Note that elements $w_{ij}$ are belonged to W and $\mu_1,\mu_2,
\dots \mu_n$ denote eigenvalues of L. 
\end{remark}
\begin{proof}
After $L=W\Lambda W^{-1}$ decomposition; equation
$L\tau^n=\tau^{n+1}$  is converted to $W\Lambda W^{-1}
\tau^n=\tau^{n+1}$ where $\Lambda $ is diagonal eigenvalue matrix of
$L$. is We only need top elements of vectors, thus each side is
multiplied by
$$
\left(
\begin{array}{cccc}
1&0&\dots&0\\
\end{array}
\right)
$$
 $(\rho\times 1)$matrix from right.  Any particular term of sequence $S_n$ satisfies
\begin{equation} \label{summat}
S_n= \left(
\begin{array}{cccc}
1&0&\dots&0\\
\end{array}
\right) (W\Lambda^{n-\rho+2}(W^{-1}\tau^{\rho-2}))
\end{equation}
 equation. Since
$$
W^{-1}\tau^{\rho-2}= \left(
\begin{array}{rrrr}
w^{-1}_{1,1}&w^{-1}_{1,2}&\dots&w^{-1}_{1,\rho}\\
w^{-1}_{2,1}&w^{-1}_{2,2}&\dots&w^{-1}_{2,\rho}\\
\dots&\dots&\ddots&\dots\\
w^{-1}_{{\rho },1}&w^{-1}_{{\rho},2}&\dots&w^{-1}_{\rho,\rho}\\
\end{array}
\right) \left(
\begin{array}{c}
S_{\rho-2}\\
S_{\rho-3}\\
\vdots\\
S_{0}\\
\pi\\
\end{array}
\right)
$$
and 
$$
W =  \left(
\begin{array}{rrrr}
w_{1,1}&w_{1,2}&\dots&w_{1,\rho}\\
w_{2,1}&w_{2,2}&\dots&w_{2,\rho}\\
\dots&\dots&\ddots&\dots\\
w_{{\rho},1}&w_{{\rho},2}&\dots&w_{\rho,\rho}\\
\end{array}
\right)
$$
and
$$
\Lambda^{n-\rho+2} =
\left(
\begin{array}{rrrr}
\mu_{1}^{n-\rho+2}&0&\dots&0\\
0&\mu_{2}^{n-\rho+2}&\ddots&\vdots\\
\vdots&\ddots&\ddots&0\\
0&\dots&0&\mu_{\rho}^{n-\rho+2}\\
\end{array}
\right)
$$
are known, (\ref{summat}) can be computed.
After (\ref{summat}) is simplified, initial expression
(\ref{eq:15}) is obtained.
\end{proof}

\section{Concluding Remarks}
Several sequences in OEIS \cite{ref2} seem identical to some multiple counting sequences. $A007910$ is a multiple counting sequence when ($x=2$, $\rho=5$) and $A077947$ when ($x=2$,$\rho=7$) which are also called Fourth and Sixth Order Jacobsthal Sequences \cite{ref8}. Additionally $A000302$ and $A093138$ are multiple counting sequences for ($x=4$, $\rho=3$) and ($x=10$, $\rho=3$) definitions. Alternative examples goes on.

Unlike prime-counting function $\Pi(x)$ which is $A006880$ of OEIS \cite{ref2}, sequences $J_n$ and $S_n$ count composites of given prime $\rho$. There is only one exception for both sequences in their initial terms. If $\rho$ is prime, $J_n$ and $S_n$ counts all composites of prime $\rho$.

As this study has generalized relationships between defined
sequences  and numbers of multiples between defined boundaries, outcomes of our research can be used in various electronics, telecommunications and
computer  science applications. Another considerable consequence of
our study might be an implementation of a new prime prediction and
factorization method due to the strong connection with Fermat's
little theorem. Distribution of prime numbers is a key concept in
present number theory and as a measure of this phenomenon, the role
of multiple counting sequences is undeniable.
 
\section{Acknowledgement}
I would like to thank Mustafa Sefa Y\"{o}r\"{u}ker, Adem \c{S}ahin, Mustafa Zengin and Ayhan Yenilmez for their helpful advices.

\bigskip
\hrule
\bigskip

\noindent 2010 {\it Mathematics Subject Classification}:
Primary 11Nxx; Secondary 11Pxx, 11Txx, 11Yxx, 11Bxx

\noindent \emph{Keywords: } 
Carmichael number, Fermat's little theorem, Binet formula, floor function, multiple-counting sequence, Fermat pseudoprime, Jacobsthal Sequence

\bigskip
\hrule
\bigskip

\noindent (Concerned with sequences
\seqnum{A001045},
\seqnum{A007910},
\seqnum{A093138},
\seqnum{A077947}, and
\seqnum{A000302}.)

\bigskip
\hrule
\bigskip

\vspace*{+.1in}
\noindent
preprint submitted to Journal of Integer Sequences.

\bigskip
\hrule
\bigskip

\noindent
Return to
\htmladdnormallink{Journal of Integer Sequences home page}{http://www.cs.uwaterloo.ca/journals/JIS/}.
\vskip .1in

\end{document}